\documentclass[11pt]{article}
\usepackage{graphicx}
\usepackage{amsmath}
\usepackage{amscd}
\usepackage{amssymb}
\usepackage{enumerate}
\usepackage{indentfirst}
\usepackage{latexsym}
\usepackage{multicol}
\usepackage{color}

\newenvironment{proof}{\textit{Proof. }}{\hfill$\Box$}
\newtheorem{Theorem}{Theorem}[section]

\newtheorem{Lemma}{Lemma}[section]
\newtheorem{Corollary}{Corollary}[section]
\newtheorem{Remark}{Remark}[section]
\newtheorem{Definition}{Definition}[section]

\newcommand{\vcg}[1]{{\pmb #1}}

\newcommand{\vr}{\varrho}
\newcommand{\vre}{\vr_\ep}

\newcommand{\vte}{\vt_\ep}
\newcommand{\vue}{\vu_\ep}

\newcommand{\vt}{\vartheta}
\newcommand{\vu}{\vc{u}}
\newcommand{\vf}{\vc{f}}

\newcommand{\vc}[1]{{\bf #1}}

\newcommand{\Div}{{\rm div}_x}
\newcommand{\dive}{{\rm div}_x}

\newcommand{\tn}[1]{\mbox {\F #1}}
\newcommand{\dx}{\,{\rm d} {x}}

\newcommand{\intO}[1]{\int_{\Omega} #1 \dx}

\newcommand{\vv}{\vc{v}}

\newcommand{\ep}{\varepsilon}
\newcommand{\Om}{\Omega}
\newcommand{\Ome}{\Omega_\varepsilon}
\def\SSS{\tn{S}}
\def\supp{\rm supp}
\font\F=msbm10 scaled 1100
\font\FF=msbm10 scaled 750
\font\FFF=msbm10 scaled 900
\newcommand{\RR}{\mbox{\FF R}}

\def\e{\varepsilon}
\def\d{\partial}
\def\a{\alpha}
\renewcommand\d{\partial}
\renewcommand\a{\alpha}

\newcommand\R{\mathbb R}

\def\g{\gamma}
\def\de{\delta}

\def\vp{\varphi}
\def\e{\varepsilon}

\newcommand{\N}{{\mathbb N}}
\newcommand{\Z}{{\mathbb Z}}

\numberwithin{equation}{section}

\newcommand\bpm{\begin{pmatrix}}
\newcommand\epm{\end{pmatrix}}
\newcommand\be{\begin{equation}}
\newcommand\ee{\end{equation}}
\newcommand\ba{\begin{equation}\begin{aligned}}
\newcommand\ea{\end{aligned}\end{equation}}
\newcommand\nn{\nonumber}

\date{}

\begin{document}


\title{Homogenization of stationary Navier--Stokes--Fourier system in domains with tiny holes}

\author{Yong Lu\thanks{Nanjing University, Department of Mathematics, 22 Hankou Road, Gulou District, 210093 Nanjing, China. Email: luyong@nju.edu.cn. Yong Lu acknowledges the support of the Recruitment Program of Global Experts of China. This work is partially supported by project ANR JCJC BORDS funded by l'ANR of France.} \and Milan Pokorn\'y \thanks{Charles University, Faculty of Mathematics and Physics, Sokolovsk\' a 83, 186 75 Praha 8, Czech Republic.  Email: pokorny@karlin.mff.cuni.cz. Milan Pokorn\'y acknowledges the support of the project 19-04243S of the Czech Science Foundation. The paper was prepared during his stay at Nanjing University and he acknowledges also this support.}
}

\maketitle

\begin{abstract}
We study the homogenization of stationary compressible Navier--Stokes--Fourier system in a bounded three dimensional domain perforated with a large number of very tiny holes. Under suitable assumptions imposed on the smallness and distribution of the holes, we show that the homogenized limit system remains the same in the domain without holes. 

\end{abstract}




\section{Introduction}
\label{s_i}

Homogenization in fluid mechanics gives rise to system of partial differential equations considered on physical domains perforated by a large number of tiny holes (obstacles). The main concern is the asymptotic behavior of the fluid flows when the size of the holes goes to zero and the number of the holes goes to infinity simultaneously.  The ratio between the diameter and mutual distance of these holes plays a crucial role. Mathematically, the goal is to describe the limit behavior of the solutions to the partial differential equations used to describe the fluid flows.  With an increasing number of holes, the fluid flow approaches an effective state governed by certain {\em homogenized} equations which are defined in homogeneous domains---domains without holes.


 For Stokes and stationary incompressible Navier--Stokes equations, Allaire \cite{ALL-NS1,ALL-NS2} (see also earlier results by Tartar \cite{Tartar1}) gave a systematic study for different sizes of holes. We recall Allaire's result in more details for domains in three dimensions. Consider a family of holes of diameter $O(\e^{\alpha})$, where $\e$ is their mutual distance.  Allaire showed that when $1\leq \alpha<3$ corresponding to the case of large holes, the limit fluid behavior is governed by the classical Darcy's law;  when $\alpha>3$ corresponding to the case of tiny holes, the equations do not change in the homogenization process and the limit problem is determined by the same system of Stokes or Navier--Stokes equations; when $\alpha=3$ corresponding to the case of critical size of holes, in the limit there yields the Brinkman's law---a damping term is added to the original system, which looks like a combination of the original Stokes or Navier--Stokes equations and the Darcy's law. Related results for the evolutionary (time-dependent) incompressible Navier--Stokes system were obtained by Mikeli\'{c} \cite{Mik} and, more recently, by Feireisl, Namlyeyeva and Ne\v{c}asov\'a \cite{FeNaNe}. We note that the holes are assumed to be periodically distributed in Allaire's results, while in \cite{FeNaNe}  more general distribution of holes was considered.

For the homogenization of compressible fluids, even under periodic setting of the distribution of holes, there are no systematic results as in the incompressible case.  The earlier results mainly focus on the specific case $\alpha=1$, meaning that the size of holes is proportional to their mutual distance. Masmoudi \cite{Mas-Hom} identified rigorously the porous medium equation and Darcy's law as a homogenization limit for the evolutionary barotropic compressible Navier--Stokes system in the case where the diameter of the holes is comparable to their mutual distance. Similar results for the full Navier--Stokes--Fourier system were obtained in \cite{FNT-Hom}.

 When $\a >1$, the perforated domain has three scales and the homogenization problem becomes quite different in compressible case. Unlike the incompressible case where one works only in $L^2$ framework, one needs to work in general $L^p$ framework for compressible case. We refer to \cite{Lu-Stokes} for more explanations.  For the case with large holes ($1<\a<3$) and the case with critical size of holes ($\a=3$), there are basically no results. While for the case with small holes ($\a>3$), the first author and his collaborators proved similar results as the incompressible setting and showed that the motion is not affected by the obstacles and the limit problem coincides with the original one: in \cite{FL1} and \cite{DFL} for stationary  compressible (isentropic) Navier--Stokes system,  in \cite{Lu-Sch} for evolutionary compressible (isentropic)  Navier--Stokes system.
  
  While, according to the authors' knowledge, there is no result in the homogenization of full compressible Navier--Stokes--Fourier system when $\a \neq 1.$ In this paper, we are working in this direction and focus on the case of small holes $\a >3$. The main new difficulties lie in obtaining uniform estimates for the temperatures and building a compatible extension of the temperatures.  Based on an idea of \cite{CoDo_88} which goes back to \cite{CiSJP_79}, we construct an extension operator which is bounded from $W^{1,2}(\Om_\e)$ to $W^{1,2}(\Om)$, and is bounded from $L^r(\Ome)$ to $L^r(\Om)$ for all $r\in [1,\infty]$. Moreover, it preserves the value in $\Ome$ and the non-negativity property of the temperature. By employing this extension operator, we proved the uniform $L^{3m}(\Om_\e)$ bound for the family of temperatures as $\e\to 0$. 
 
In the sequel, we use $C$ to denote a positive constant independent of $\e$, for which the value may differ from line to line. 

\section{Problem formulation, main results}
\label{s_p}

\subsection{Perforated domain}
\label{sub_p}

We study the steady compressible Navier--Stokes--Fourier system in a domain perforated with many tiny holes. Let $\ep>0$ be a small number which is used to measure the mutual distance between the holes. We assume that our domain
\be\label{domain}
\Omega_\ep = \Omega \setminus \bigcup_{n=1}^{N(\ep)} \overline T_{n,\ep},
\ee
where $\Omega \subset \R^3$ is a bounded $C^2$-domain and
$
\{T_{n,\ep}\}_{n=1}^{N(\ep)}
$ are $C^2$-domains of the diameter comparable to $\ep^\alpha$ for some $\alpha \geq 1$ such that there exist $\delta_0$, $\delta_1$ and $\delta_2$ positive for which
\be\label{ass-holes}
T_{n,\ep} = x_{\ep,n} + \ep^\alpha T_{n,1}^0 \subset B_{\delta_0\ep^\alpha}(x_{n,\ep}) \subset B_{2\delta_0\ep^\alpha}(x_{n,\ep})\subset B_{\delta_1\ep}(x_{n,\ep}) \subset B_{\delta_2\ep}(x_{n,\ep}) \subset \Omega.
\ee
We assume that the balls $B_{\delta_2\ep}(x_{n,\ep})$ centred at $x_{\ep,n}$ with diameter $\delta_2 \ep$ are pairwise disjoint and we assume that the domains $\{T_{n,1}^0\}_{n=1}^{N(\ep)}$ are uniformly $C^2$-domains.  The former in fact gives an upper limit on the number of the holes as $N(\ep) \sim  \ep^{-3}$. Note, however, that we do not assume any periodicity for the distribution of the holes, just certain uniform behavior expressed above.

\subsection{The model}
\label{sub_m}

We consider the steady compressible Navier--Stokes--Fourier system which describes the steady flow of compressible heat conducting Newtonian fluid in perforated domain $\Omega_{\e}$ given by \eqref{domain} and \eqref{ass-holes}. The purpose is to study the homogenization of the system as $\e\to 0.$ The system reads
\begin{equation} \label{bal_mass}
\Div (\vr\vu) = 0,
\end{equation}
\begin{equation} \label{bal_mom}
\Div (\vr\vu\otimes \vu) + \nabla p(\vr,\vt) -\Div \tn{S}(\vt,\nabla \vu) =  \vr \vc{f},
\end{equation}
\begin{equation} \label{bal_en}
\Div \big(\vr E\vu + p\vu -\tn{S}(\vt,\nabla \vu)\vu+ \vc{q}\big) = \vr \vc{f}\cdot\vc{u}.
\end{equation}

We complete the system by the boundary conditions on $\partial \Omega_\ep$
\begin{equation} \label{bc_vel}
\vc{u} = \vc{0},
\end{equation}
\begin{equation} \label{bc_temp}
\vc{q} \cdot \vc{n} + L(\vt-\vt_0) = 0
\end{equation}
and by prescribing the total mass
\begin{equation} \label{giv_mass}
\int_{\Omega_\ep} \vr \dx = M_\e >0.
\end{equation}

The unknown quantities are the density $\vr$: $\Omega_\ep \to \R_{\geq 0}$, the velocity $\vu$: $\Omega_\ep \to \R^3$ and the temperature $\vt$: $\Omega \to \R_+$. We are not able to conclude that the density is positive, while we can ensure that the temperature is positive, at least for the weak solutions presented below.

Furthermore, we have to specify the constitutive relations in the equations above. We first assume that the pressure
\begin{equation} \label{press}
p(\vr,\vt) = \vr^\gamma + \vr \vt.
\end{equation}
Here we require $\gamma >2$.  Note that we could also consider more general  pressure forms (as, e.g., in \cite{NoPo_JDE}). However, our main concern is the homogenization for the system, so we will not work much on the direction of weakening the assumptions of the pressure term. It would, moreover, technically complicate the paper.
  Next, the stress tensor corresponds to the compressible Newtonian fluid 
\begin{equation} \label{stress}
\tn{S}(\vt,\nabla \vu) = \mu(\vt)\Big(\nabla \vu + \nabla^{\rm T} \vu -\frac 23 \Div\vu \,\tn{I}\Big) + \nu(\vt) \Div\vu \,\tn{I},
\end{equation}
where the viscosity coefficients are continuous functions of the temperature on $\R_+$, the shear viscosity $\mu(\cdot)$ is moreover globally Lipschitz continuous, and
\begin{equation} \label{visc}
C_1(1+\vt) \leq \mu(\vt) \leq C_2(1+\vt), \qquad 0\leq \nu(\vt) \leq C_2(1+\vt).
\end{equation}
The heat flux is given by the Fourier law
\begin{equation} \label{heat}
\vc{q}(\vt,\nabla \vt)= -\kappa(\vt) \nabla \vt,
\end{equation}
where the heat conductivity is assumed to satisfy
\begin{equation} \label{ht_cond}
C_3(1+\vt^m) \leq \kappa(\vt) \leq C_4(1+\vt^m)
\end{equation}
for some positive $m$. In our case we require at least $m>2$. The  total energy is given as
$$
E = e + \frac 12 |\vu|^2,
$$
and the specific internal energy $e$ fulfils the Gibbs relation
\begin{equation}\label{Gibbs}
\frac{1}{\vt}\Big(D e + p(\vr,\vt)D\Big(\frac{1}{\vr}\Big)\Big) = Ds(\vr,\vt)
\end{equation}
which leads to
\begin{equation} \label{int_e}
e(\vr,\vt) = c_v \vt + \frac{\vr}{\gamma-1},
\end{equation}
where the unknown function of temperature was set, for simplicity, as a linear one. Moreover, we can view (\ref{Gibbs}) as the definition of a new thermodynamic potential, the specific entropy, which is given uniquely up to an additive constant. It fulfils formally the balance of entropy
$$
\Div\Big(\vr s\vu + \frac{\vc{q}}{\vt}\Big) = \sigma = \frac{\tn{S}:\nabla \vu}{\vt} - \frac{\vc{q}\cdot \nabla \vt}{\vt^2}.
$$
Finally, the data are the external force $\vc{f}$, the given mass $M_{\e}>0$, the external temperature $\vt_0>0$ prescribed on $\partial \Omega_\ep$, and the positive constant $L$.

The existence of strong (or classical) solutions to this system of PDEs under hypothesis made above is out of reach of nowadays mathematics unless we require ``smallness" of the data. We therefore work with weak solutions which are known to exist for the above relations in the range of $m$'s and $\gamma$'s much wider than we need for our purpose of the homogenization study.

\subsection{Weak formulation in perforated domains}
\label{w}

We are in position to present the weak formulation of our problem in $\Omega_\ep$. Below we assume that all functions are sufficiently regular, i.e., all integrals written down are finite.

The weak formulation of the continuity equation reads
\begin{equation} \label{w_cont}
\int_{\RR^3} \vr \vu\cdot \nabla \psi \dx = 0
\end{equation}
for all $\psi \in C^1_c(\R^3)$, where $\vr$ and $\vu$ are extended by zero outside of $\Omega_\e$. Moreover, we need to work with a renormalized form of this equation
\begin{equation} \label{w_cont_ren}
\int_{\RR^3} \Big(b(\vr)\vu \cdot \nabla \psi + (b(\vr)-\vr b'(\vr))\Div\vu \psi\Big)\dx =0
\end{equation}
for all $\psi \in C^1_c(\R^3)$ and all  $b\in C^1([0,\infty))$ such that $b'\in C_0([0,\infty))$, and both $\vr$ and $\vu$ are extended by zero outside of $\Omega_\e$. We remark that this restriction on $b$ could be relaxed, see Remark \ref{rem:ren} below.

The weak formulation of the momentum equation with the homogeneous Dirichlet boundary conditions has the form
\begin{equation} \label{w_mom}
\int_{\Omega_\ep} \Big(-\vr (\vu\otimes\vu):\nabla \vcg{\varphi} - p(\vr,\vt) \Div \vcg{\varphi} + \tn{S}(\vt,\nabla \vu):\nabla \vcg{\varphi}\Big) \dx = \int_{\Omega_\ep} \vr\vc{f}\cdot \vcg{\varphi}\dx
\end{equation}
for all $\vcg{\varphi} \in C^1_c(\Omega_\ep;\R^3)$.

The weak formulation of the total energy balance reads
\begin{equation} \label{w_ener}
-\int_{\Omega_\ep} \Big(\vr E \vu  + p(\vr,\vt)\vu - \tn{S}(\vt,\nabla \vu)\vu + \vc{q}\Big) \cdot \nabla \psi \dx + \int_{\partial \Omega_\ep} L(\vt-\vt_0)\psi \,{\rm d}S = \int_{\Omega_\ep} \vr\vc{f}\cdot \vu \psi \dx
\end{equation}
for all $\psi \in C^1(\overline{\Omega_\ep})$. Furthermore, we also have the entropy inequality
\begin{equation} \label{w_entr}
\begin{aligned}
&\int_{\Omega_\ep} \Big(\frac{\tn{S}(\vt,\nabla\vu)}{\vt} - \frac{\vc{q}\cdot \nabla \vt}{\vt^2}\Big)\psi \dx + \int_{\partial \Omega_\ep} \frac{L \vt_0 }{\vt}\psi \,{\rm d}S
\leq  L \int_{\partial \Omega_\ep}  \psi \,{\rm d}S + \int_{\Omega_\ep} \Big(-\frac{\vc{q}\cdot \nabla \psi}{\vt} -\vr s(\vr,\vt) \vu\cdot \nabla \psi\Big) \dx
\end{aligned}
\end{equation}
for all $\psi \in C^1(\overline{\Omega_\ep})$, non-negative.

\begin{Definition} \label{d_weak_sol}
We say that the triple ($\vr$, $\vu$, $\vt$), $\vr \geq 0$ and $\vt>0$ a.e. in $\Omega_\ep$, is a renormalized weak entropy solution to our problem (\ref{bal_mass})--(\ref{int_e}), if $\vr \in L^\gamma(\Omega_\ep)$, $\vu \in W^{1,2}_0(\Omega_\ep;\R^3)$, $\vt^{\frac{m}{2}}$ and $\log \vt \in W^{1,2}(\Omega_\ep)$ such that $\vr |\vu|^3$, $|\tn{S}(\vt,\nabla\vu)\vu|$ and $p(\vr,\vu)|\vu| \in L^1(\Omega_\ep)$ and the relations (\ref{w_cont}), (\ref{w_cont_ren}), (\ref{w_mom}), (\ref{w_ener}) and (\ref{w_entr}) are fulfilled with test functions specified above.
\end{Definition}

For fixed $\ep>0$ we have the following existence result, see \cite{NoPo_JDE} for detailed proof.
\begin{Theorem}\label{exis_ep}
Let $\vc{f} \in L^\infty(\Omega;\R^3)$, $\vt_0 \in L^1(\partial\Omega_\ep)$, $\vt_0 \geq T_0>0$ a.e. on $\partial \Omega_\ep$, $L>0$, $M_{\e}>0$. Let $\gamma >\frac 53$ and $m>1$. Then there exists a renormalized weak entropy solution ($\vr$, $\vu$, $\vt$) to our problem (\ref{bal_mass})--(\ref{int_e}) in the sense of Definition \ref{d_weak_sol}.
\end{Theorem}

\subsection{Main result}
\label{sub_mr}

We now investigate the limit passage $\ep \to 0^+$. In what follows, we will consider a sequence of weak entropy solutions to our problem from Theorem \ref{exis_ep}, denoted as ($\vre$, $\vue$, $\vte$). We will show that, extending suitably the sequence to the whole domain $\Omega$, it is bounded in certain spaces ($\vre$ in $L^{\gamma +\Theta}(\Omega)$ for some $\Theta = \Theta(\gamma,m)>0$, $\vue$ in $W^{1,2}_0(\Omega;\R^3)$ and $\vte$ in $W^{1,2}(\Omega) \cap L^{3m}(\Omega)$). We show that the corresponding weak limit of the extension sequence, taking a subsequence if necessary, solves the same stationary Navier--Stokes--Fourier system in the weak sense in $\Omega$. More precisely, our main result reads

\begin{Theorem} \label{t_main}
Let $\vc{f} \in L^\infty(\Omega;\R^3)$, $M_\e>0$ with $\sup_\e M_\e  = M_1 < \infty$, $\inf_\e M_\e  = M_0 > 0$, $L>0$ and let $\vt_0\geq T_0>0$ in $\Omega$ be defined so that it has finite $L^q$-norm over arbitrary  smooth two-dimensional surface with finite surface area contained in $\Omega$ for some $q > 1$. Let ($\vre$, $\vue$, $\vte$) denote the corresponding renormalized weak entropy solution to (\ref{bal_mass})--(\ref{int_e}) for fixed $\ep >0$, extended suitably to the whole $\Omega$ as shown in Section \ref{s_e} below, for which in particular the extensions preserve their values in $\Omega_\e$.  Let $\alpha>3$, $m>2$ and $\gamma>2$ fulfil $\alpha> \max\{\frac{2\gamma-3}{\gamma-2}, \frac{3m-2}{m-2}\}$. Then there holds the uniform bound
\be\label{uni-bound}
\|\vre\|_{L^{\gamma +\Theta}(\Omega)} + \|\vue\|_{W^{1,2}_0(\Omega)} + \|\vte\|_{W^{1,2}\cap L^{3m}(\Omega)} \leq C,
\ee
where $\Theta := \min\Big\{2\gamma-3, \gamma \frac{3m-2}{3m+2}\Big\}$.
Moreover,  the corresponding weak limit of the sequence for $\ep \to 0^+$ is a renormalized weak solution to problem (\ref{bal_mass})--(\ref{int_e}) in $\Omega$, i.e., it fulfils the continuity equation in the weak and renormalized sense, the mass balance and the total energy balance in the weak sense in $\Omega$, and $\vr \geq 0$ and $\vt >0$ a.e. in $\Omega$.
\end{Theorem}
Note that we do not know whether the entropy inequality \eqref{w_entr} is also fulfilled in the limit. This is an interesting open question. Next, we could skip the requirement that the infimum over all total masses is strictly positive. However, if the limit total mass would be zero, then the solution is trivial ($\varrho =0$, $\vv=\mathbf{0}$ with some temperature distribution) and we prefer to avoid this case. 

\medskip

We give a remark concerning the renormalized equation:
\begin{Remark}\label{rem:ren}
By DiPerna--Lions' transport theory (see {\rm \cite[Section II.3]{DiP-L}} and the modification in {\rm \cite[Lemma 3.3]{N-book}}), for any $\rho\in L^\beta(\Omega),~\beta\geq 2,\ \vv\in W^{1,2}_0(\Omega;\R^3)$, where $\Omega\subset  \R^3$ is a bounded domain of class $C^{0,1}$, such that
$$
\Div (\rho \vv)=0 \quad \mbox{in}\quad \mathcal{D}'(\Omega),
$$
there holds the renormalized equation
$$
 \Div \big(b(\rho)\vv\big)+\big(\rho b'(\rho) -b(\rho)\big)\Div \vv=0, \quad \mbox{in}\ \mathcal{D}'(\R^3),
$$
for any $b\in C^0([0,\infty))\cap C^1((0,\infty))$ satisfying
\be\label{b-pt1}
b'(s)\leq C\, s^{-\lambda_0} \ \mbox{for}\  s\in (0,1],\quad \ b'(s)\leq C\,s^{\lambda_1} \ \mbox{for}\  s\in [1,\infty)
\ee
with
\be\label{b-pt2}
C>0,\quad \lambda_0<1, \quad -1<\lambda_1 \leq \frac{\beta}{2}-1,
\ee
provided $\rho$ and $\vv$ have been extended to be zero outside $\Omega$.

\end{Remark}
From Remark \ref{rem:ren} and estimate \eqref{uni-bound} we see that the continuity equation is satisfied in the renormalized sense with $b$ satisfying weaker assumptions \eqref{b-pt1} and \eqref{b-pt2}.

\section{Uniform bounds}
\label{s_ub}

By virtue of the weak entropy formulation we can deduce several bounds for our sequence in $\Omega_\ep$. We use the weak formulation of the entropy inequality with test function $\psi\equiv 1$ and get that (note that $\sum_{n=1}^{N(\ep)}|\partial T_{n,\ep}|\sim \ep^{2\alpha-3}$ and that $\alpha >3$)
\begin{equation} \label{b1}
\int_{\Omega_\ep} \Big(\frac{\tn{S}(\vte,\nabla\vue):\vue}{\vte} + \frac{(1+\vte^m) |\nabla \vte|^2}{\vte^2}\Big)\dx +  \int_{\partial \Omega_\ep} \frac{L \vt_0 }{\vte} \,{\rm d}S \leq C.
\end{equation}
Further, let us take $\psi \equiv 1$ also in the total energy balance. It gives
\begin{equation} \label{b2}
\int_{\partial \Omega_\ep} L \vte \,{\rm d}S  \leq C\Big(1 + \int_{\Omega_\ep} \vre|\vue|\dx\Big).
\end{equation}
Hence we have, due to the form of the stress tensor and the Korn inequality,\footnote{Recall that for arbitrary $\vu \in W^{1,2}_0(\Omega_\ep;\mbox{\FFF R}^3)$ if follows by direct integration by parts that
$$
\|\vue\|_{W^{1,2}_0(\Omega_\ep)}^2 \leq C \|\nabla \vue\|_{L^2(\Omega_\ep)}^2 \leq C \int_{\Omega_\ep} \frac{\mbox{\FFF{S}}(\vte,\nabla \vue):\nabla \vue}{\vte}\dx.
$$}
\begin{equation} \label{b3}
\begin{aligned}
&\|\vue\|_{W^{1,2}_0(\Omega_\ep)} + \|\nabla \log \vte\|_{L^2(\Omega_\ep)} + \|\nabla |\vte|^{\frac m2}\|_{L^2(\Omega_\ep)} + \Big\|\frac 1{\vte}\Big\|_{L^1(\partial \Omega)} \leq C \\
&\|\vte\|_{L^1(\partial \Omega)} \leq C\big(1+ \|\vre\|_{L^{\frac 65}(\Omega_\ep)}\big).
\end{aligned}
\end{equation}
Sobolev embedding implies that the family $\vue$ is bounded in $L^6(\Omega_\ep)$.  Note that the bounds in \eqref{b3} imply that the norm $\|\vte\|_{L^{3m}(\Omega_\ep)}$ is finite (controlled by the $L^{\frac 65}$-norm of the density), however, we do not know whether it is uniform with respect to $\ep$. Nonetheless, we will verify this in the next section independently of the results which follow. For now, we use the fact that $\|\vte\|_{L^{3m}(\Omega_\ep)}$ is bounded as $\e \to 0^+$ and this fact will be proved later on, independently of the results in this section below. 

In order to estimate the density, we use the result of Diening, Feireisl and Lu. It reads (see \cite[Theorem 2.3]{DFL}): 
\begin{Theorem} \label{t_DFL}
Let a family of domains $\Omega_\ep$ be defined by \eqref{domain} and \eqref{ass-holes}. Then there exists a family of linear operators
$$
\mathcal{B}_\ep\colon L^q_0(\Omega_\ep) \to W^{1,q}_0(\Omega_\ep;\R^3), \quad 1<q<\infty,
$$
such that for arbitrary $f \in L^q_0(\Omega_\ep)$ it holds
$$
\begin{aligned}
&\Div \mathcal{B}_\ep(f) = f \quad \text{ a.e. in }\Omega_\ep, \\
&\|\mathcal{B}_\ep(f)\|_{W^{1,q}_0(\Omega_\ep)} \leq C \big(1 + \ep^{\frac{(3-q)\alpha-3}{q}}\big) \|f\|_{L^q(\Omega_\ep)},
\end{aligned}
$$
where the constant $C$ is independent of $\ep$. Here  $L^q_0(\Omega_\ep)$ denote the set of $L^{q}(\Om_{\e})$ functions which have zero mean value.  
\end{Theorem}
  In bounded Lipschitz domain the existence of Bogovskii operator is well-known (see \cite{BOG}, \cite{GAL}).  While the operator norm depends on the Lipschitz character of the domain, and for the perforated domain $\Om_{\e}$, its Lipschitz norm is unbounded  as $\e \to 0^{+}$ due to the presence of small holes.  The above result gives a Bogovskii type operator on perforated domain $\Om_{\e}$ with a precise dependency of the operator norm on $\e$. For some $\e$ and $q$, such a Bogovskii-type operator is uniformly bounded.

Using this result we may get the following estimate of the density:
\begin{Lemma} \label{l_est_dens}
Let $\gamma >2$, $m>2$ and $\alpha> \max \big\{\frac{2\gamma-3}{\gamma-2}, \frac{3m-2}{m-2}\big\}$. Suppose that $\|\vte\|_{L^{3m}(\Omega_\ep)}$ is bounded.  Then the sequence $\{\vre\}$ is bounded
in $L^{\gamma + \Theta}(\Omega_\ep)$, where
\be\label{Theta}
\Theta = \min\Big\{2\gamma-3, \gamma \frac{3m-2}{3m+2}\Big\}.
\ee
\end{Lemma}

\begin{proof}
We use the version of the Bogovskii operator from Theorem \ref{t_DFL}.
We use as a test function in (\ref{w_mom}) the following function
$$\vcg{\varphi} := {\mathcal{B_\ep}}\big(\vre^\Theta - \langle \vre^\Theta \rangle \big), \quad \langle \vre^\Theta \rangle := \frac{1}{|\Omega_\ep|}\int_{\Omega_\ep} \vre^\Theta \dx,$$
where $\Theta>0$ to be determined. Recall that
\be\label{c0}
\|\nabla \vcg{\varphi}\|_{L^q(\Omega_\ep)} \leq C(\e,q) \|\vre^\Theta\|_{L^q(\Omega)}, \quad \mbox{with} \ C(\e,q) : =  C \big(1 + \ep^{\frac{(3-q)\alpha-3}{q}}\big).
\ee
 We see that $C(\e, q)$ is independent of $\ep$ provided $1<q<3$ satisfying $(3-q)\alpha-3\geq 0$. We get
\begin{equation}\label{c1}
  \int_{\Omega_\ep} p(\vre,\vte)\vre^\Theta \dx
	= \int_{\Omega_\ep}\Big(p(\vre,\vte) \frac{1}{|\Omega_\ep|}\int_{\Omega_\ep} \vre^\Theta \dx- \vre(\vue\otimes\vue):\nabla \vcg{\varphi}
	+ \tn{S}(\vte,\nabla \vue):\nabla \vcg{\varphi} - \vre\vc{f}\cdot\vcg{\varphi}\Big)\dx.
\end{equation}
We now estimate the right hand-side of \eqref{c1} term by term. We start with the two most restrictive terms which give the limit on the exponent $\Theta$.
First, we consider
\ba\label{density-est-1}
  \Big|\int_{\Omega_\ep}\vre(\vue\otimes\vue):\nabla\vcg{\varphi}\dx\Big|
	&\le \|\vue\|_{L^6(\Omega_\ep)}^2\|\vre\|_{L^{\gamma+\Theta}(\Omega_\ep)}
	\|\nabla\vcg{\varphi}\|_{L^{q_{1}}(\Omega_\ep)} \\
	&\le C(\e,q_1) \|\vue\|_{L^6(\Omega_\ep)}^2\|\vre\|_{L^{\gamma+\Theta}(\Omega_\ep)}
	\|\vre^\Theta\|_{L^{q_1}(\Omega_\ep)}\\
& \le C(\e,q_1) \|\vue\|_{L^6(\Omega_\ep)}^2\|\vre\|_{L^{\gamma+\Theta}(\Omega_\ep)}
	\|\vre\|_{L^{q_1\Theta}(\Omega_\ep)}^\Theta,
\ea
where
\be\label{density-est-2}
\frac{1}{q_1} = 1 - \frac{1}{3} - \frac{1}{\gamma+\Theta}, \quad C(\e,q_1) = C \big(1 + \ep^{\frac{(3-q_1)\alpha-3}{q_1}}\big).
\ee
We want to choose $\Theta$ as large as possible such that $\vr_\e$ enjoys as high as possible integrability. For this reason, we choose $\Theta$ such that $q_1 \Theta = \gamma+ \Theta.$  Together with \eqref{density-est-2}, we end up with
$$
\Theta = \Theta_1 : = 2\gamma-3>1, \quad  q_1 = \frac{3(\gamma-1)}{2\gamma-3}, \quad (3-q_1)\alpha-3  = 3\left[ \frac{\gamma-2}{2\gamma -3} \alpha - 1\right].
$$
 Hence, under the condition
$$
\frac{\gamma-2}{2\gamma -3} \alpha \geq 1,
$$
we have $C(q_1,\e)$ independent of $\e$, and by using \eqref{b3}, we deduce from \eqref{density-est-1} that
$$
\Big|\int_{\Omega_\ep}\vre(\vue\otimes\vue):\nabla\vcg{\varphi}\dx\Big| \le C \|\vre\|_{L^{\gamma+\Theta_1}(\Omega_\ep)}^{\Theta_1+1},
$$
 where in particular $C$ is independent of $\e$ and $\Theta_1+1 < \gamma  + \Theta_1.$

\medskip

 Next we calculate
\ba\label{density-est-4}
  \Big|\int_{\Omega_\ep}\tn{S}(\vte,\nabla\vue):\nabla\vcg{\varphi}\dx\Big|
	&\le C\big(\e,\frac{6m}{(3m-2)}\big)(1+\|\vte\|_{L^{3m}(\Omega_\ep)}) \|\nabla \vue\|_{L^2(\Omega_\ep)} \|\nabla\vcg{\varphi}\|_{L^{\frac{6m}{(3m-2)}}(\Omega_\ep)} \\
	&\le C\big(\e,\frac{6m}{(3m-2)}\big)\|\nabla \vue\|_{L^2(\Omega_\ep)} \|\vre^{\Theta}\|_{L^{\frac{6m}{(3m-2)}}(\Omega_\ep)} \\
	&\le C\big(\e,\frac{6m}{(3m-2)}\big)  \|\nabla \vue\|_{L^2(\Omega_\ep)}  \|\vre\|_{L^{\frac{6m \Theta}{(3m-2)}}(\Omega_\ep)}^{\Theta},
\ea
where $C\big(\e,\frac{6m}{(3m-2)}\big)$ is defined in the same manner as $C(\e,q)$ in \eqref{c0}. We choose $\Theta$ such that $\frac{6m \Theta}{(3m-2)} = \gamma + \Theta$ and we end up with
$$
\Theta = \Theta_2 : = \frac{\gamma (3m-2)}{(3m+2)}>1, \quad  \left(3- \frac{6m}{(3m-2)}\right)\alpha-3  = 3\left[ \frac{m-2}{3m -2} \alpha - 1\right].
$$
We see  $C\big(\e,\frac{6m}{(3m-2)}\big)$ is independent of $\e$ provided $ \frac{m-2}{3m -2} \alpha \geq 1$, and moreover by using \eqref{b3} we deduce from \eqref{density-est-4} that
\ba\label{density-est-6}
  \Big|\int_{\Omega_\ep}\tn{S}(\vte,\nabla\vue):\nabla\vcg{\varphi}\dx\Big| \le C  \|\vre\|_{L^{\gamma+ \Theta_2}(\Omega_\ep)}^{\Theta_2}.
\ea

\medskip

From \eqref{density-est-1}--\eqref{density-est-6}, we see that by choosing
$$
\Theta : = \min\{\Theta_1, \Theta_2\} =  \min\Big\{2\gamma-3, \gamma \frac{3m-2}{3m+2}\Big\}>1, \quad \alpha > \max\Big\{\frac{2\gamma-3}{\gamma-2}, \frac{3m-2}{m-2}\Big\}>3,
$$
there holds
\ba\label{density-est-7}
\Big|\int_{\Omega_\ep}\vre(\vue\otimes\vue):\nabla\vcg{\varphi}\dx\Big| +  \Big|\int_{\Omega_\ep}\tn{S}(\vte,\nabla\vue):\nabla\vcg{\varphi}\dx\Big| \le C\left(1 + \|\vre\|_{L^{\gamma+\Theta}(\Omega_\ep)}^{\Theta+1}\right).
\ea

\medskip

 Further, due to the restriction $\Theta \leq 2\gamma-3$ we have $\frac 32 \Theta < \gamma + \Theta$; and note also that $\|\vcg{\varphi}\|_{L^3(\Omega_{\e})}\leq C\|\nabla \vcg{\varphi}\|_{L^{3/2}(\Omega_\ep)}$ holds with a constant independent of $\ep$ due to the zero trace of $\vcg{\varphi}$ on $\partial \Omega_\ep$; we thus deduce
\ba\label{density-est-8}
  \Big|\int_{\Omega_\ep}\vre\vc{f}\cdot\vcg{\varphi}\dx\Big|
	&\le \|\vc{f}\|_{L^\infty(\Omega_\ep)}\|\vre\|_{L^{3/2}(\Omega_\ep)}
	\|\vcg{\varphi}\|_{L^{3}(\Omega_\ep)} \\
	&\le C\|\vc{f}\|_{L^\infty(\Omega_\ep)}\|\vre\|_{L^{3/2}(\Omega_\ep)}
	\|\nabla \vcg{\varphi}\|_{L^{3/2}(\Omega_\ep)} \\
	&\leq C \|\vre\|_{L^{\gamma + \Theta}(\Omega)}^{1+\Theta}.
\ea

Finally, due to the form of the pressure
\ba\label{density-est-9}
  \int_{\Omega_\ep} p(\vre, \vte) \dx \frac{1}{|\Omega_\ep|}\int_{\Omega_\ep} \vre^\Theta \dx
	& \leq C \int_{\Omega_\ep}(\vre \vte+\vre^\gamma)\dx\int_{\Omega_\ep} \vre^\Theta \dx \\
&\leq C \left(  \|\vte\|_{L^{6}(\Omega)}  \|\vre\|_{L^{\frac{6}{5}}(\Omega)} + \|\vre\|_{L^{\gamma}(\Omega)}^{\gamma} \right)\|\vre\|_{L^{\Theta}(\Omega)}^{\Theta}\\
&\leq C \left( 1 + \|\vre\|_{L^{\gamma}(\Omega)}^{\gamma}\|\vre\|_{L^{\Theta}(\Omega)}^{\Theta} \right)\\
&\leq C \left( 1 + \|\vre\|_{L^{\gamma+\Theta}(\Omega)}^{\lambda} \right),
\ea
for some $\lambda < \gamma+\Theta$. In the last inequality in \eqref{density-est-9} we used the interpolation between the $L^1$ and $L^{\gamma +\Theta}$ norms to control $\|\vre\|_{L^{\gamma}(\Omega)}$ and $\|\vre\|_{L^{\Theta}(\Omega)}$,  as we control the $L^1$ norm of the density (i.e., the total mass). Here we assume that $\vte$ is bounded in $L^{2m}(\Ome)\subset L^6(\Omega_\e)$. 
	

Collecting the estimates in \eqref{density-est-7}, \eqref{density-est-8} and \eqref{density-est-9}, we derive from \eqref{density-est-1} that
$$
 \|\vre\|_{L^{\gamma+\Theta}(\Omega)}^{\gamma+\Theta} \leq C \left( 1 + \|\vre\|_{L^{\gamma+\Theta}(\Omega)}^{\lambda} \right), \quad \mbox{for some $1<\lambda < \gamma+ \Theta$}.
$$
 This implies our desired estimate of $\|\vre\|_{L^{\gamma+\Theta}(\Omega)}$, which is uniform with respect to $\ep$ provided we have the uniform bound for $\vte$ in $L^{3m}(\Omega_\ep)$.
\end{proof}

Hence, combining the result of the previous lemma with available estimates, under the assumptions in Lemma \ref{l_est_dens}, we end up with 
$$
\|\vue\|_{W^{1,2}(\Omega_\ep)} + \|\vre\|_{L^{\gamma +\Theta}(\Omega_\ep)} + \|\log \vte\|_{W^{1,2}(\Omega_\ep)} + \|\nabla\vte^{\frac m2}\|_{L^{2}(\Omega_\ep)} \leq C
$$
and
$$
\|\vte\|_{L^1(\partial \Omega)} + \|\vt^{-1}_\ep\|_{L^1(\partial \Omega)} \leq C.
$$
Note, however, that we still need to know that the $L^{3m}(\Omega_\ep)$ norm of $\vte$ is controlled uniformly with respect to $\ep$. We will show this in the next section.

\section{Extensions of functions}
\label{s_e}

We can now extend our triple of functions to the whole $\Omega$. For the density and the velocity we simply extend the functions by zero.  After this extension we still have
$$
\begin{aligned}
\|\vu_\ep\|_{W^{1,2}_0(\Omega)} =\|\vu_\ep\|_{W^{1,2}_0(\Omega_{\e})} \leq C, \quad \|\vue\|_{L^6(\Omega)} \leq C\|\nabla \vue\|_{L^2(\Omega)} = C \|\nabla \vue\|_{L^2(\Omega_\ep)} \leq C, \quad  \|\vre\|_{L^{\gamma+\Theta}(\Omega)} \leq C.
\end{aligned}
$$
However, the issue with the temperature is more delicate.
Note that we in general even do not know whether it holds
$$
\|\vte\|_{L^{3m}(\Omega_\ep)} \leq C
$$
with a constant $C$ independent of $\ep$, as the domains $\Omega_\ep$ are not uniformly Lipschitz with respect to $\ep \to 0^+$. We start with one more general result which is due to Conca and Dorato \cite{CoDo_88} which uses even an older idea of Cioranescu and Paulin \cite{CiSJP_79}. Since we need a slightly stronger information from their result, we present the full proof of the result.

\begin{Lemma} \label{l_ext_CD}
Let $\Om_{\e}$ be given by \eqref{domain} and \eqref{ass-holes}. There exists an extension operator $E_\ep$: $W^{1,2}(\Omega_\ep) \to W^{1,2}(\Omega)$ such that for each $\varphi \in W^{1,2}(\Omega_\ep)$,
$$
\begin{aligned}
& E_\ep \varphi(x) = \varphi(x), \quad x\in \Omega_\ep,\\
& \|\nabla E_\ep \varphi \|_{L^2(T_{n,\ep})} \leq C \|\nabla \varphi\|_{L^2(B_{2\delta_0\ep^\alpha}(x_{n,\ep})\setminus T_{n,\ep})}
\end{aligned}
$$
and hence $\|\nabla E_\ep \varphi\|_{L^2(\Omega)} \leq C \|\nabla \varphi\|_{L^2(\Omega_\ep)}$.
Moreover, for all $1\leq q \leq \infty$,
$$
\|E_\ep \varphi\|_{L^q(T_{n,\ep})} \leq C \|\varphi\|_{L^q(B_{2\delta_0\ep^\alpha}(x_{n,\ep})\setminus T_{n,\ep})}.
$$
The constant $C$ is independent of $\ep$ and $n$.

 Furthermore, there is an extension operator $\tilde E_\e: W^{1,2}_{\geq 0}(\Omega_\ep)\to W^{1,2}_{\geq 0}(\Omega)$ such that the above properties are also satisfied. Here $W^{1,2}_{\geq 0}(\Omega_\ep)$ denotes the set of nonnegative functions in $W^{1,2}(\Omega_\ep)$.
\end{Lemma}

\begin{proof}
We start as in the proof of \cite[Lemma A.1]{CoDo_88}. We namely show the existence of the extension operator from $W^{1,2}(B_{2\delta_0\ep^\alpha}(x_{n,\ep})\setminus T_{n,\ep}) \to W^{1,2}(B_{2\delta_0\ep^\alpha}(x_{n,\ep}))$ satisfying the properties above. To this aim, recall the assumption on the distribution of the holes in \eqref{ass-holes} and let $\varphi \in W^{1,2}(B_{2\delta_0}(0)\setminus T_{n,1}^0)$. We write for $x \in B_{2\delta_0}(0)\setminus T^0_{n,1}$
$$
\varphi = M\varphi + \psi,
$$
where $M\varphi := \frac{1}{|B_{2\delta_0}(0)\setminus T_{n,1}^0|}\int_{B_{2\delta_0}(0)\setminus T_{n,1}^0} \varphi \dx$ (the mean value) and $M\psi =0$.  Since  $T_{n,1}^0$ are uniformly $C^2$-domains, then for each $n$, there exists an extension operator $\tilde S$ from $W^{1,2}(B_{2\delta_0}(0)\setminus T_{n,1}^0)$ to $W^{1,2}(B_{2\delta_0}(0))$ such that for each $\psi \in W^{1,2}(B_{2\delta_0}(0)\setminus T_{n,1}^0)$, there holds
\ba\label{pt-tildeS-1}
&\tilde S \psi(x) = \psi(x), \quad x\in B_{2\delta_0}(0)\setminus T_{n,1}^0,\\
&\|\tilde S \psi\|_{W^{1,2}(B_{2\delta_0}(0))} \leq C \|\psi\|_{W^{1,2}(B_{2\delta_0}(0)\setminus T_{n,1}^0)},\\
&\|\tilde S \psi\|_{L^{r}(B_{2\delta_0}(0))} \leq C \|\psi\|_{L^{r}(B_{2\delta_0}(0)\setminus T_{n,1}^0)}, \quad \forall\, 1\leq r \leq \infty,
\ea
where $C$ is independent of $n$ and $r$.  We apply $\tilde S$ on the function $\psi$ in $B_{2\delta_0}(0)\setminus T_{n,1}^0$. Since the mean value of $\psi$ is zero in $B_{2\delta_0}(0)\setminus T_{n,1}^0$,  we get
\be\label{pt-tildeS-2}
\|\tilde S \psi\|_{W^{1,2}(B_{2\delta_0}(0))} \leq C \|\psi\|_{W^{1,2}(B_{2\delta_0}(0)\setminus T_{n,1}^0)} \leq C \|\nabla \psi\|_{L^{2}(B_{2\delta_0}(0)\setminus T_{n,1}^0)}  = C \|\nabla \varphi\|_{L^{2}(B_{2\delta_0}(0)\setminus T_{n,1}^0)}.
\ee
We now set for $x \in B_{2\delta_0}(0)$
\be\label{def-S}
S \varphi := M\varphi + \tilde S \psi.
\ee
By \eqref{pt-tildeS-1} and \eqref{pt-tildeS-2}, we still keep
\ba\label{pt-tildeS-3}
&\tilde S \varphi(x) = \varphi(x), \quad x\in B_{2\delta_0}(0)\setminus T_{n,1}^0,\\
&\|\nabla \tilde S \varphi\|_{L^{2}(B_{2\delta_0}(0))} \leq C \|\nabla \varphi\|_{L^{2}(B_{2\delta_0}(0)\setminus T_{n,1}^0)},\\
&\|\tilde S \varphi\|_{L^{q}(B_{2\delta_0}(0))} \leq C \|\varphi\|_{L^{q}(B_{2\delta_0}(0)\setminus T_{n,1}^0)}, \quad \forall\, 1\leq q \leq \infty,
\ea
where the constant $C$ can be taken independent of $n$.

\medskip

We are now ready to define our desired extension operator $E_\e$.   For each $\varphi \in W^{1,2}(B_{2\delta_0\ep^\alpha}(x_{n,\ep})\setminus T_{n,\ep})$, define
$$
\tilde \varphi (y) := \varphi(x_{n,\e} + \e^\alpha y), \quad \forall \, y\in B_{2\delta_0}(0)\setminus T_{n,1}^0.
$$
Then $\tilde \varphi \in W^{1,2}(B_{2\delta_0}(0)\setminus T_{n,1}^0)$. We apply the extension operator $S$ defined through \eqref{def-S} to $\tilde \varphi$ and obtain $S \tilde \varphi \in W^{1,2}(B_{2\delta_0}(0))$. Finally we define the extension operator $E_{\e}$ as
$$
E_\ep \varphi(x) := (S \tilde \varphi) \left(\frac{x - x_{n,\e}}{\e^\alpha}\right).
$$
Clearly $E_\ep \varphi \in  W^{1,2}(B_{2\delta_0\ep^\alpha}(x_{n,\ep}))$ and $E_\e \varphi = \varphi$ in  $ B_{2\delta_0\ep^\alpha}(x_{n,\ep})\setminus T_{n,\ep}$ due to  the first property in \eqref{pt-tildeS-3}. By the second property in \eqref{pt-tildeS-3}, we then calculate
\begin{align*}
\int_{B_{2\delta_0\ep^\alpha}(x_{n,\ep})} |\nabla_x E_\ep\varphi|^2 \dx & =  \int_{B_{2\delta_0\ep^\alpha}(x_{n,\ep})} \ep^{-2\alpha}\left|(\nabla_y S\tilde \varphi) \left(\frac{x - x_{n,\e}}{\e^\alpha}\right)\right|^2 \,{\rm d} x \\
& = \ep^{\alpha} \int_{B_{2\delta_0}(0)}  \left|(\nabla_y S\tilde \varphi) (y)\right|^2 \,{\rm d} y\\
&  \leq C \ep^\alpha\int_{B_{2\delta_0}(0)\setminus T_{n,1}^0} |\nabla_y \tilde \varphi|^2 \,{\rm d}y \\
& = C  \int_{B_{2\delta_0\ep^\alpha}(x_{n,\ep})\setminus T_{n,\ep}} |\nabla_x \varphi|^2 \,{\rm d}x.
\end{align*}
The third property in \eqref{pt-tildeS-3} yields
$$
\|E_\ep \varphi\|_{L^q(T_{n,\ep})} \leq C \|\varphi\|_{L^q(B_{2\delta_0\ep^\alpha}(x_{n,\ep})\setminus T_{n,\ep})}, \quad \forall\, 1\leq q \leq \infty.
$$
 To obtain the extension from $W^{1,2}(\Omega_\ep)$ to $W^{1,2}(\Omega)$, we simply sum the extensions for $n=1$ to $N(\ep)$.

 \medskip
 
 To finish the proof, we assume that the function $\varphi$ is nonnegative. It is sufficient to modify the construction by taking
$$
\tilde E_\ep \varphi := \max\{0,E_\ep \varphi\},
$$
and recall that
$$
\|\nabla \tilde E_\ep \varphi\|_{L^2(B_{2\delta_0\ep^\alpha}(x_{n,\ep}))} \leq \|\nabla E_\ep \varphi\|_{L^2(B_{2\delta_0\ep^\alpha}(x_{n,\ep}))}
$$
and
$$
\|\tilde E_\ep \varphi\|_{L^q(B_{2\delta_0\ep^\alpha}(x_{n,\ep}))} \leq \|E_\ep \varphi\|_{L^q(B_{2\delta_0\ep^\alpha}(x_{n,\ep}))}, \quad \forall\, 1\leq q\leq\infty.
$$

\end{proof}

\begin{Remark}
Indeed, in the previous lemma we can replace the $L^2$ norm of the gradient by an arbitrary $L^p$ norm with $1\leq p\leq \infty$,  as well as instead of three space dimensions we can work in $\R^d$, $d\geq 2$.  However, we do not need all these generalizations in this paper.
\end{Remark}

We apply this extension $\tilde E_\e$ on $\vt_\ep$ and we have the following result:
\begin{Lemma} \label{l_ext_temp}
The extended temperature $\tilde E_\ep \vt_\ep$ is uniformly with respect to $\ep$ bounded in $W^{1,2}(\Omega)$ and in $L^{3m}(\Omega)$.
\end{Lemma}

\begin{proof}First of all, since $\vte \in W^{1,2}(\Omega_\e)$ and $\vte > 0 $ a.e. in $\Omega_\e$, we have $\tilde E_\ep \vt_\ep\in W^{1,2}(\Omega)$ and $\tilde E_\e \vte \geq 0 $ a.e. in $\Omega$. The point is to have uniform control of the norms.

Recall that $m>2$.  By the fact that $\vte^2 \leq C(1+\vte^m) $ and by \eqref{b1}, we have $\nabla \vt_\ep$ bounded in $L^2(\Omega_\ep)$. By Lemma \ref{l_ext_CD},  the extended function $\tilde E_\ep \vt_\ep$ has the gradient controlled uniformly in $L^2(\Omega_\ep)$ and the function itself belongs to $W^{1,2}(\Omega)$. Moreover, as we control the $L^1$ norm of $\vt_\ep$ over the boundary $\partial \Omega$ and this value is not influenced by the extension, we in fact know that $\tilde E_\ep \vt_\ep$ has uniformly controlled $L^6$ norm over $\Omega$. Therefore, the $L^6$ norm over $\Omega_\ep$ is also bounded uniformly with respect to $\ep$, because $\tilde E_\e \vte$ coincides with $\vte$ in $\Omega_\e$.

 Assume for a moment that $m\leq 12$ where we already have $\vte^{\frac m2}$ bounded in $L^1(\Omega_\e)$. Together with \eqref{b2}, we have $\vte^\frac{m}{2}$ uniformly bounded in $W^{1,2}(\Omega_\e)$.  Then we apply the extension from Lemma \ref{l_ext_CD} on $\vt^{\frac m2}$. Exactly as above it yields that the $L^{3m}$ norm of the extension over the whole $\Omega$ is bounded uniformly with respect to $\ep$. This means also the $L^{3m}$ norm of $\vt_\ep$ must be bounded uniformly with respect to $\ep$ over $\Omega_\ep$. This allows us to apply the claim on the estimate of the $L^q$ norm of the extension from Lemma \ref{l_ext_CD} with $q=3m$ for $\vte$.\footnote{Remark that for the extensions of $\vte$ and $\vte^{\frac m2}$ in general $(\tilde E_\ep\vte)^{\frac m2} \neq \tilde E_\ep \vte^{\frac m2}$ in $\Omega \setminus\Omega_\ep$.}

 For $m>12$, we proceed by induction. From the previous step, we have $\vte$ uniformly bounded in $L^{36}(\Omega_\e)$. Simply copying the argument of the proof of the case $2<m\leq 12$, we can cover all $2< m \leq 36$. Then we can go further and cover all $m>2$.

\end{proof}

The last information we need is a version of the trace theorem. Indeed, in a fixed domain, the trace of $\vte$ belongs to $L^{2m}(\partial\Omega_\ep)$. The question is whether we can control its norm uniformly with respect to $\ep$.  The following lemma gives a quantitative estimate on each $\d T_{n,\e}$.

\begin{Lemma} \label{trace}
Under the assumptions stated in Theorem \ref{t_main} there holds
$$
\|\vte\|_{L^{2m}(\partial T_{n,\ep})}^{2m} \leq C \big(\|\nabla |\vte|^{\frac m2}\|_{L^2(B_{2\delta_0\ep^\alpha}(x_{n,\ep})\setminus T_{n,\ep})}^{2} + \|\vte\|_{L^{3m}(B_{2\delta_0\ep^\alpha}(x_{n,\ep})\setminus T_{n,\ep})}^{3m}+ \|\vte\|_{L^{3m}(B_{2\delta_0\ep^\alpha}(x_{n,\ep})\setminus T_{n,\ep})}^{2m}\big),
$$
where the constant $C$ is independent of $\ep$ and $n$.
\end{Lemma}

\begin{proof}
Recalling the standard proof of the trace theorem for Sobolev functions (see, e.g., \cite{Evans_book}),  one can arrive at (by partition of unity and smooth approximation) the following inequality
$$
\int_{\partial T_{n,\ep}} |\vte|^{2m} \,{\rm d}S \leq C\int_{B_{2\delta_0\ep^\alpha}(x_{n,\ep})\setminus T_{n,\ep})}\big| \nabla (\varphi_\ep |\vte|^{2m})\big| \dx,
$$
where the function $\varphi_\ep$ is a non-negative and  smooth cut-off function which equals to 1 on $\partial T_{n,\ep}$ and vanishes near $\partial B_{2\delta_0 \ep^\alpha}(x_{n,\ep})$, hence its gradient is bounded by $C\ep^{-\alpha}$.
Recall that we used the fact that the domain $T_{n,\ep}$ is close to a ball with diameter $\ep^{\alpha}$, uniformly with respect to $\ep$.
To finish the proof, we need to estimate the right hand-side of the inequality. We calculate
\begin{align*}
&\int_{B_{2\delta_0\ep^\alpha}(x_{n,\ep})\setminus T_{n,\ep})}\big|\nabla (\varphi_\ep |\vte|^{2m})\big|  \dx \\
&\quad \leq  \int_{B_{2\delta_0\ep^\alpha}(x_{n,\ep})\setminus T_{n,\ep})}|\nabla \varphi_\ep| |\vte|^{2m} \dx + \int_{B_{2\delta_0\ep^\alpha}(x_{n,\ep})\setminus T_{n,\ep}}\varphi_\ep \big|\nabla |\vte|^{2m}\big| \dx \\
&\quad \leq  C \ep^{-\alpha} \int_{B_{2\delta_0\ep^\alpha}(x_{n,\ep})\setminus T_{n,\ep})} |\vte|^{2m} \dx + C \int_{B_{2\delta_0\ep^\alpha}(x_{n,\ep})\setminus T_{n,\ep})}\big|\nabla |\vte|^{\frac m2}\big| |\vte|^{\frac{3m}{2}} \dx \\
& \quad \leq  C \Big(\int_{B_{2\delta_0\ep^\alpha}(x_{n,\ep})\setminus T_{n,\ep})} |\vte|^{3m} \dx\Big)^{\frac 23} \\
& \qquad+ C \Big(\int_{B_{2\delta_0\ep^\alpha}(x_{n,\ep})\setminus T_{n,\ep})}\big|\nabla |\vte|^{\frac m2}\big|^2 \dx\Big)^{\frac 12} \Big(\int_{B_{2\delta_0\ep^\alpha}(x_{n,\ep})\setminus T_{n,\ep})}|\vte|^{3m} \dx\Big)^{\frac 12}.
\end{align*}
This implies that
\ba\label{trace-theta-2}
&\int_{\partial T_{n,\ep}} |\vte|^{2m} \,{\rm d}S  \leq  C \Big(\int_{B_{2\delta_0\ep^\alpha}(x_{n,\ep})\setminus T_{n,\ep})} |\vte|^{3m} \dx\Big)^{\frac 23} \\
& \quad + C \int_{B_{2\delta_0\ep^\alpha}(x_{n,\ep})\setminus T_{n,\ep})}\big|\nabla |\vte|^{\frac m2}\big|^2 \dx + C\int_{B_{2\delta_0\ep^\alpha}(x_{n,\ep})\setminus T_{n,\ep})}|\vte|^{3m} \dx,
\ea
which leads to the desired inequality.

\end{proof}

A direct corollary from Lemma \ref{trace} is the following trace estimate on the whole boundary of the holes. We will see that this estimate is not uniform bounded in $\e$. However, we obtain an explicit dependency on $\e$. This will be needed later when passing limit in the energy balance equation.
\begin{Corollary}\label{cor-trace}Under the assumptions stated in Theorem \ref{t_main} there holds
$$\|\vte\|_{L^{2m}(\cup_{n=1}^{N(\ep)} \d T_{n,\ep})} \leq C \e^{-\frac{1}{2m}}.$$
\end{Corollary}
\begin{proof}
By Lemma \ref{trace} (by \eqref{trace-theta-2} specifically), we have
\begin{align*}
\int_{\cup_{n=1}^{N(\ep)}\d T_{n,\ep}} |\vte|^{2m} \,{\rm d}S   & =\sum_{n=1}^{N(\ep)} \int_{\d T_{n,\ep}} |\vte|^{2m} \,{\rm d}S \\
 & \leq  C \sum_{n=1}^{N(\ep)}  \Big(\int_{B_{2\delta_0\ep^\alpha}(x_{n,\ep})\setminus T_{n,\ep})} |\vte|^{3m} \dx\Big)^{\frac 23} \\
& \quad + C \sum_{n=1}^{N(\ep)} \int_{B_{2\delta_0\ep^\alpha}(x_{n,\ep})\setminus T_{n,\ep})}\big|\nabla |\vte|^{\frac m2}\big|^2 \dx + C \sum_{n=1}^{N(\ep)}\int_{B_{2\delta_0\ep^\alpha}(x_{n,\ep})\setminus T_{n,\ep})}|\vte|^{3m} \dx\\
& \leq C \left(\sum_{n=1}^{N(\ep)}  \Big(\int_{B_{2\delta_0\ep^\alpha}(x_{n,\ep})\setminus T_{n,\ep})} |\vte|^{3m} \dx\Big)\right)^{\frac{2}{3}} \left(\sum_{n=1}^{N(\ep)}  1 \right)^{\frac{1}{3}} \\
&\quad + C  \int_{\Omega_\e}\big|\nabla |\vte|^{\frac m2}\big|^2 \dx + C \int_{\Omega_\e}|\vte|^{3m} \dx\\
& \leq C \e^{-1},
\end{align*}
where we used the boundedness of $\|\nabla |\vte|^{\frac m2}\|_{L^2(\Omega_\e)}$ and $\|\vte\|_{L^{3m}(\Omega)}$.  Our desired result follows immediately.

\end{proof}

To summarize, starting from the solution sequence $\vre, \vue, \vte$, we find a sequence of extension of functions, still denoted by $\vre, \vue, \vte$, such that the extensions are defined on the whole $\Omega$, are equal to the original functions on $\Omega_\ep$, $\vre$ and $\vue$ are zero in $\Omega \setminus \Omega_\ep$, $\vte\geq 0$  in $\Omega$, and they satisfy the following uniform bounds with respect to $\ep$
\begin{equation}\label{b7}
\begin{aligned}
\|\vue\|_{W^{1,2}(\Omega)} \leq C, \quad \|\vre\|_{L^{\gamma +\Theta}(\Omega)} \leq C, \quad \|\vte\|_{W^{1,2}(\Omega)} + \|\vte\|_{L^{3m}(\Omega)} \leq C,
\end{aligned}
\end{equation}
where $\Theta$ is from Lemma \ref{l_est_dens}. Moreover, $\vte$ has well defined trace on $\partial T_{n,\ep}$ and Lemma \ref{trace} and Corollary \ref{cor-trace} provide us a control of its norm.

\section{Limit passage}
\label{s_lp}

To conclude, we need to show that, up to a remainder which goes to zero when $\ep \to 0^+$, the functions fulfil the weak formulations of the continuity, momentum and energy equations in the whole $\Omega$. This will be the goal of the following two subsections (for the continuity equations there is nothing to do). To show that the weak limits of the sequences form in fact a weak solution to the steady compressible Navier--Stokes--Fourier system we will have to show the strong convergence of the density sequence. This is, however, nowadays standard in the mathematical fluid mechanics of compressible fluids. Last but not least, we have to check the the limit of the temperatures is in fact positive a.e. in $\Omega$ since the extensions could become zero on a nontrivial set, however, this set is contained in $\bigcup_{n=1}^{N(\ep)} T_{n,\ep}$ which is a set whose measure is of order $O(\ep^{3(\alpha-1)})$ when $\ep \to 0^+$.

\medskip

First of all, from the uniform bound in \eqref{b7}, up to a selection of subsequences, we have the following convergence results:
\ba\label{conv-1}
&\vue \to \ \vu \ \mbox{weakly in} \  W_0^{1,2}(\Omega;\R^3), \quad \vue \to \ \vu \ \mbox{strongly in} \  L^{r}(\Omega;\R^3), \ \mbox{for all} \ 1\leq r < 6,\\
& \vre \to \ \vr \ \mbox{weakly in} \  L^{\gamma + \Theta}(\Omega), \\
&\vte \to \ \vt \ \mbox{weakly in} \  W^{1,2}(\Omega), \quad \vte \to \ \vt \ \mbox{strongly in} \  L^{r}(\Omega), \ \mbox{for all} \ 1\leq r < 3m.
\ea

\subsection{Limit passage in the energy equation}
\label{sub_lpee}

We now want to show that if we plug into the weak formulation of our problem our extended functions, we can in fact write the problem as a weak formulation of the total energy balance on $\Omega$ (but for the functions $(\vre,\vue,\vte)$ extended to $\Omega$) and a small remainder which goes to zero if $\ep \to 0^+$. Moreover, we show that passing with $\ep \to 0^+$, we get the weak formulation of the total energy balance in $\Omega$ for the limit functions $(\vr,\vu,\vt)$. Here we, however, need to show the strong convergence of the sequence of densities which is not obvious, but nowadays standard. This is postponed to the last section. 

Recall that $\vue = \mathbf{0}$ on $\Omega\setminus \Omega_\e$. We then can rewrite the weak formulation of the energy balance as follows:
$$
\begin{aligned}
&-\int_{\Omega} \Big(\vre \big(e(\vre,\vte)+ \frac 12 |\vue|^2\big) \vue  + p(\vre,\vte)\vue - \tn{S}(\vte,\nabla \vue)\vue - \kappa(\vte)\nabla \vte\Big) \cdot \nabla \psi \dx \\
&\quad + \int_{\partial \Omega} L(\vte-\vt_0)\psi \,{\rm d}S - \int_{\Omega} \vre\vc{f}\cdot \vue \psi \dx\\
& = -\int_{\Omega\setminus\Omega_\ep} \kappa(\vte)\nabla\vte\cdot \nabla \psi \dx -  \int_{\cup_{n=1}^{N(\ep)}\d T_{n,\ep}} L(\vte-\vt_0)\psi \,{\rm d}S \\
&=: I_1 + I_2
\end{aligned}
$$
for each $\psi \in C^1(\overline{\Omega})$. Let us show that both integrals on the right hand-side disappear when $\ep \to 0^+$. By H\"older's inequality, we have, as $\ep \to 0^+$,
$$
|I_1| \leq C \|\nabla \psi\|_{L^\infty} (1+\|\vte\|^m_{L^{3m}(\Omega\setminus\Omega_\ep)}) \|\nabla \vte\|_{L^2(\Omega\setminus\Omega_\ep)} |\Omega\setminus\Omega_\ep|^{\frac 16} \to 0.
$$
Using Corollary \ref{cor-trace} and the fact that the sequence $\|\vt_0\|_{L^q(\d \Omega_\e)}$ is bounded with respect to $\e$ for some $q>1$, together with the fact that $\a>3$ and $m>2$,  we have
\begin{align*}
|I_2| & \leq C \Big(\|\vte\|_{L^{2m}(\cup_{n=1}^{N(\ep)} \d T_{n,\ep})} \Big|\bigcup_{n=1}^{N(\ep)} \d T_{n,\ep}\Big|^{\frac{2m-1}{2m}} + \|\vt_0\|_{{L^{q}(\cup_{n=1}^{N(\ep)}\d T_{n,\ep})}} \Big|\bigcup_{n=1}^{N(\ep)} \d T_{n,\ep}\Big|^{\frac{q-1}{q}}\Big)\\
& \leq C \e^{-\frac{1}{2m}} \e^{(2\a -3)\frac{2m-1}{2m}} + C \e^{(2\a -3)\frac{q-1}{q}} \\
& \leq C \e^{\frac{(2m-1)(2\a - 3)-1}{2m}} + C  \e^{(2\a -3)\frac{q-1}{q}}  \to 0, \quad \mbox{as $\e \to 0^+.$}
\end{align*}
Hence, passing to the limit on the left hand-side, recalling that $\vue\to \vu$ strongly in $L^r(\Omega;\R^3)$ with all $1\leq r <6$ and $\vte\to \vt$ strongly in $L^r(\Omega)$ with all $1\leq r <3m$, we get (recall that as $\alpha >\frac{3m-2}{m-2}$, we know $(2m-1)(2\alpha-3)>1$)
\begin{equation} \label{w_ener_2}
-\int_{\Omega} \Big(\big(\overline{\vr e (\vr,\vt)}+ \vr \frac 12 |\vu|^2\big) \vu  + \overline{p(\vr,\vt)}\vu - \tn{S}(\vt,\nabla \vu)\vu - \kappa(\vt)\nabla \vt\Big) \cdot \nabla \psi \dx + \int_{\partial \Omega} L(\vt-\vt_0)\psi \,{\rm d}S = \int_{\Omega} \vr\vc{f}\cdot \vu \psi \dx,
\end{equation}
where we used the notation $\overline{g(\vr)}$ being a weak limit of $g(\vre)$ in some suitable $L^r(\Omega)$ space. To conclude that we get the total energy balance for the limit functions we need to show that the sequence of densities $\vre$ converges in fact strongly to $\vr$ at least in $L^1(\Omega)$. This will be the aim of the last subsection.

\medskip

We finish this subsection by the following result.
\begin{Lemma} \label{temp_pos}
The limit temperature $\vt$ is positive a.e. in $\Omega$.
\end{Lemma}

\begin{proof}
We first apply Lemma \ref{l_ext_CD} on the sequence $\log \vte$ (the operator $E_\ep$).\footnote{Recall that in $\Omega\setminus\Omega_\ep$ in general $\log \tilde E_\ep(\vte) \neq E_\ep(\log \vte)$.} Since both $\int_{\partial \Omega} \vte\,{\rm d}S$ and $\int_{\partial \Omega} \vte^{-1}\,{\rm d}S$ are bounded uniformly with respect to $\ep$, we see that the sequence $E_\ep (\log \vte)$ is bounded in $W^{1,2}(\Omega)$ and in particular, up to a subsequence, $E_\ep(\log \vte) \to z$ in $L^r(\Omega)$ for all $1\leq r <6$ and a.e. in $\Omega$.
In particular, $z >-\infty$ a.e. in $\Omega$.

Next, we take a specific sequence of $\ep_l\to 0^+$ such that $\ep_l \leq \frac{1}{l}$ for all $l \in \tn{N}$. Note that  the three-dimensional Lebesgue measure
$$
\Big| \bigcup_{n=1}^{N(\ep_l)} T_{n,\ep_l}\Big| \leq \frac{C}{l^{3(\alpha-1)}},
$$
and since $\alpha >2$, the series
$
\sum_{l=1}^\infty  \frac{1}{l^{3(\alpha-1)}}
$
is convergent. Let us denote for $l_0 \in \tn{N}$
$$
D_{l_0} = \bigcup_{l=l_0}^\infty \bigcup_{n=1}^{N(\ep_l)} T_{n,\ep_l}.
$$
Then for any $\delta>0$ there exists $l_0 \in \tn{N}$ such that the three-dimensional Lebesgue measure of $D_{l_0}$ is smaller than $\delta$.

Let us assume that the limit temperature constructed in \eqref{conv-1} is zero on a set of positive three-dimensional Lebesgue measure, say of measure $\delta_0>0$. We take $l_0$ corresponding to $\delta_0/2$ from the above construction, where we chose a subsequence from $\ep \to 0^+$ such that $\ep_l \leq \frac{1}{l}$ with $l\geq l_0$. Since we know that our sequence of temperatures $\tilde E_\ep \vt_{\ep_l}$ converges strongly in $L^q(\Omega)$ for any $q<3m$ and a.e. in $\Omega$, it also converges a.e. in $\Omega \setminus D_{l_0}$. Hence we know that $\log(\tilde E_\ep \vte)$ converges strongly in $L^q(\Omega \setminus D_{l_0})$ for some $q\geq 1$ and a.e. in $\Omega \setminus D_{l_0}$ to $\log \vt$, e.g., by virtue of Vitali's convergence theorem. But then also $\ln \vt = z >-\infty$ a.e. in $\Omega \setminus D_{l_0}$. This means that the limit temperature $\vt$ could be zero at most on $D_{l_0}$ together with a set of  measure zero. Thus $\vt$ cannot be zero on a set of measure $\delta_0$ which leads to a contradiction.
\end{proof}

\subsection{Limit passage in the continuity and the momentum equation}
\label{sub_lpcme}

First, recall that the continuity equation is satisfied in the weak and renormalized sense \eqref{w_cont} and \eqref{w_cont_ren} for all $\psi\in C_0^1(\R^d) $ with $b\in C^0([0,\infty))\cap C^1((0,\infty))$ satisfying \eqref{b-pt1} and \eqref{b-pt2}. Passing $\e \to 0$ and applying \eqref{conv-1} gives that
\be\label{con-con-1}
\Div(\vr \vu) = 0  \quad \mbox{holds in}\ \mathcal{D}'(\R^3)
\ee
and
$$
 \Div \big(\overline{b(\vr)}\vu\big)+ \overline{\big(\vr b'(\vr) -b(\vr)\big)\Div \vu} =0 \quad \mbox{holds in}\ \mathcal{D}'(\R^3),
$$
where we used the common notation $\overline{g(u)}$ denoting the weak limit of $g(u_n)$ for a nonlinear function $g$.  Moreover, by \eqref{conv-1}, \eqref{con-con-1} and Remark \ref{rem:ren}, we have (recall that $\gamma\geq 2$)
\be\label{con-con-3}
 \Div \big(b(\vr)\vu\big)+\big(\vr b'(\vr) -b(\vr)\big)\Div \vu=0, \quad \mbox{holds in}\ \mathcal{D}'(\R^3),
\ee
for any $b\in C^0([0,\infty))\cap C^1((0,\infty))$ satisfying \eqref{b-pt1} and \eqref{b-pt2}.

\medskip

It is more complicated to deduce a modified momentum system in homogeneous domain $\Omega$, due to the choice of test functions: the original momentum equations are satisfied in $\Omega_\e$ and one should choose $C_c^1(\Ome)$ test function, while our target equations are defined in $\Om$ and one should choose $C_c^1(\Om)$ test functions. We will employ the argument in \cite{DFL} and prove the following lemma:
\begin{Lemma}\label{lem:momentum}
Under the assumptions in  Theorem \ref{t_main},  there holds
 \be\label{eq-monmentum}
\dive(\vr_\e \vu_\e\otimes \vu_\e)+\nabla p(\vr_\e,\vte) - \dive\SSS(\vte,\nabla  \vu_\e) =  \vr_\e \vf +{\bf r}_\e,\quad \mbox{in}\ \mathcal{D}'(\Om),
 \ee
where the distribution ${\bf r}_\e$ is small in the following sense:
\be\label{est-re}
|\langle {\bf r}_\e,\vcg{\varphi} \rangle_{\mathcal{D}'(\Om),\mathcal{D}(\Om)}|\leq C\, \e^{\de_1} \big(\|\nabla \vcg{\varphi}\|_{L^{\frac{3(\g+\Theta)}{2(\g+\Theta)-3}+\delta_0}(\Om)} + \|\vcg{\varphi}\|_{L^{r_1}(\Om)}\big),
\ee
for all $\vcg{\varphi} \in C^\infty_c(\Om;\R^3)$, where $\Theta$ is given by \eqref{Theta},  $\de_0>0$ is chosen such that \eqref{est-Ie1-2} or \eqref{est-Ie1-2-2} is satisfied, $1<r_1<\infty$ is determined by \eqref{def-de-r1} and $\de_1>0$ is defined in \eqref{def-de1} later on.
\end{Lemma}

\begin{proof} From the assumptions on the holes in \eqref{ass-holes}, we can find a sequence of smooth functions $g_\e\in C^\infty(\Om)$ such that
$$
0\leq g_\e \leq 1, \ \  g_\e =0 \ \mbox{on}\ \bigcup_{n=1}^{N(\e)}  T_{n,\e},\quad g_\e=1 \  \mbox{in} \ \Om \setminus \bigcup_{n=1}^{N(\e)} B_{2\de_0\e^\a}(x_{\e,n}),\quad \|\nabla g_\e\|_{L^\infty(\Om)}\leq C\, \e^{-\alpha}.
$$
Then for each $1\leq r\leq \infty$, there holds
\be\label{def-ge2}
\|1-g_\e\|_{L^r(\Om)}\leq C\, \e^{\frac{3(\alpha-1)}{r}},\quad \|\nabla g_\e\|_{L^r(\Omega)}\leq C\ \e^{\frac{3(\alpha-1)}{r}-\alpha}.
\ee

Let $\vcg{\varphi} \in C_c^\infty (\Om; \R^3)$. Then $\vcg{\varphi} g_\e \in C_c^\infty (\Om_\e; \R^3)$  is a good test function for the momentum equations \eqref{bal_mom} in $\Om_\e$. Direct calculation gives
\ba\label{wk-eq-tvu}
&\intO{ \Big(\vr_\e  (\vu_\e\otimes  \vu_\e):\nabla \vcg{\varphi}  + p( \vr_\e, \vte) \,\dive \vcg{\varphi} - \SSS(\vte, \nabla  \vu_\e):\nabla \vcg{\varphi} +  \vr_\e  \vf \cdot \vcg{\varphi} \Big)  }\\
&\quad =\int_{\Om_\e}\Big( \vr_\e  (\vu_\e\otimes  \vu_\e):\nabla (\vcg{\varphi}  g_\e)+ p( \vr_\e)\, \dive (\vcg{\varphi}  g_\e)-\SSS(\vte, \nabla  \vu_\e):\nabla (\vcg{\varphi}  g_\e) +  \vr_\e \vf \cdot (\vcg{\varphi}  g_\e) \Big)\,\dx + I_\e\\
&\quad =I_\e,\nn
\ea
where $I_\e:=\sum_{j=1}^4 I_{j,\e}$ with:
\begin{align*}
&I_{1,\e}:=  \int_{\Om}   \Big(\vr_\e   (\vu_\e\otimes   \vu_\e):(1-g_\e)\nabla \vcg{\varphi}  -  \vr_\e   (\vu_\e\otimes   \vu_\e):(\nabla g_\e\otimes  \vcg{\varphi} ) \Big)\,\dx,\\
&I_{2,\e}:= \int_{\Om} \Big( p(  \vr_\e, \vte)(1-g_\e)\dive \vcg{\varphi}  - p(\vr_\e,\vte) \nabla g_\e\cdot \vcg{\varphi}\Big) \,\dx,\\
&I_{3,\e}:=  \int_{\Om} \Big(-\SSS(\vte, \nabla   \vu_\e):(1-g_\e)\nabla \vcg{\varphi}  +\SSS(\vte, \nabla \vu_\e):(\nabla g_\e\otimes  \vcg{\varphi})\Big)\,\dx,\\
&I_{4,\e}:= \int_{\Om}    \vr_\e  \vf \cdot (1-g_\e)\vcg{\varphi} \,\dx.
\end{align*}

For $I_{1,\e}$ we estimate
\ba\label{est-Ie1}
|I_{1,\e}| &\leq  C\, \|\vr_\e\|_{L^{\g + \Theta}(\Om)} \| \vu_\e\|_{L^6(\Om)}^2 \big( \|(1-g_\e)\nabla\vcg{\varphi} \|_{L^{\frac{3(\g+\Theta)}{2(\g+\Theta)-3}}(\Om)}+ \|\nabla g_\e\otimes \vcg{\varphi} \|_{L^{\frac{3(\g+\Theta)}{2(\g+\Theta)-3}}(\Om)}\big)\\
&\quad \leq C\, \big(\|1-g_\e\|_{L^{r_1}(\Om)}\|\nabla\vcg{\varphi} \|_{L^{\frac{3(\g+\Theta)}{2(\g+\Theta)-3}+\delta_0}(\Om)} + \|\nabla g_\e\|_{L^{\frac{3(\g+\Theta)}{2(\g+\Theta)-3}+ \de_0}(\Om)} \|\vp\|_{L^{r_1}(\Om)} \big),\nn
\ea
where
\be\label{def-de-r1}
0<\de_0<1, \quad 1<r_1<\infty,\quad \frac{1}{r_1}+ \left(\frac{3(\g+\Theta)}{2(\g+\Theta)-3}+\delta_0\right)^{-1}=\frac{2(\g+\Theta)-3}{3(\g+\Theta)}.
\ee
By virtue of \eqref{def-ge2}, we have
\be\label{est-Ie1-0}
\|1-g_\e\|_{L^{r_1}(\Om)}\leq C\, \e^{\frac{3(\alpha-1)}{r_1}}, \quad \|\nabla g_\e\|_{L^{\frac{3(\g+\Theta)}{2(\g+\Theta)-3}+\de_0}(\Om)}\leq C\, \e^{3(\alpha-1)\left(\frac{3(\g+\Theta)}{2(\g+\Theta)-3}+\delta_0\right)^{-1}-\alpha}.
\ee
By the definition of $\Theta$ in Theorem \ref{t_main} (or in Lemma \ref{l_est_dens}), we will calculate the sign of the power to $\e$ in \eqref{est-Ie1-0} for two cases.

The first case is
$$
\Theta = \min\Big\{2\gamma-3, \gamma \frac{3m-2}{3m+2}\Big\} = 2\gamma-3.
$$
Then
\be\label{est-Ie1-1}
3(\alpha-1)\left(\frac{3(\g+\Theta)}{2(\g+\Theta)-3}\right)^{-1}-\alpha = 3(\alpha-1)\left(\frac{3\g-3}{2\g-3}\right)^{-1}-\alpha = \frac{\a\g-2\a-2\g+3}{\g-1}>0,
\ee
where we used the condition
$$
\alpha> \max\Big\{\frac{2\gamma-3}{\gamma-2}, \frac{3m-2}{m-2}\Big\} \geq \frac{2\gamma-3}{\gamma-2}
$$
 which implies
$$
\a\g-2\a-2\g+3>0.
$$
Then by \eqref{est-Ie1-0} and \eqref{est-Ie1-1}, we can choose $\de_0>0$ small enough such that
\be\label{est-Ie1-2}
3(\alpha-1)\left(\frac{3(\g+\Theta)}{2(\g+\Theta)-3} + \de_{0}\right)^{-1}-\alpha  =:h_1(\de_0)>0.
\ee
We finally obtain in this case
$$
|I_{1,\e}| \leq C\, \e^{\de_1} \big(\|\nabla\vcg{\varphi} \|_{L^{\frac{3\g-3}{2\g-3}+\delta_0}(\Om)} + \|\vcg{\varphi}\|_{L^{r_1}(\Om)}\big),
$$
where
$$
 \de_1:=\min\left\{\frac{3(\alpha-1)}{r_1},h_1(\de_0)\right\}>0
$$
 with $\de_0>0$ is chosen such that \eqref{est-Ie1-2} is satisfied and $1<r_1<\infty$ is determined by \eqref{def-de-r1}.

\medskip

The second case is
\be\label{th-case2}
\Theta = \min\Big\{2\gamma-3, \gamma \frac{3m-2}{3m+2}\Big\} = \gamma \frac{3m-2}{3m+2}.
\ee
Then
\begin{align*}
3(\alpha-1)\left(\frac{3(\g+\Theta)}{2(\g+\Theta)-3}\right)^{-1}-\alpha & = 3(\alpha-1)\frac{4\g m - (3 m + 2)}{6 m \g} - \alpha  \\
& = \frac{\a (2\g m - 3m - 2 ) - 4\g m + 3m +2}{2 m \g }\\
& = : h_2(\a).
\end{align*}
Since $\g>2, m>2$, then
\be\label{h2-1}
h_2'(\a) = \frac{2\g m - 3 m -2}{2 m \g} > \frac{m -2}{2 m \g} >0,
\ee
which means that $h$ is strictly increasing in $\a$. Moreover,
\be\label{h2-2}
h_2\Big(\frac{3m-2}{m-2}\Big) = \frac{(m + 2) \g - (3m +2)}{\g(m-2)}.
\ee
Recalling in this case \eqref{th-case2}, we have
\ba\label{h2-3}
2\gamma-3 \geq \gamma \frac{3m-2}{3m+2}& \Longleftrightarrow 2\gamma -  \gamma \frac{3m-2}{3m+2} \geq 3 
\\
&\Longleftrightarrow  \gamma \frac{m+2}{3m+2} \geq 1 \Longleftrightarrow  \gamma \geq \frac{3m+2}{m+2}.
\ea
By \eqref{h2-2} and \eqref{h2-3}, we obtain in case \eqref{th-case2} that
$$
h_2\Big(\frac{3m-2}{m-2}\Big) \geq 0.
$$
Hence, by \eqref{h2-1} we deduce
$$
h_2(\a) = 3(\alpha-1)\left(\frac{3(\g+\Theta)}{2(\g+\Theta)-3 }\right)^{-1}-\alpha > 0
$$
for all
$$
\alpha> \max\Big\{\frac{2\gamma-3}{\gamma-2}, \frac{3m-2}{m-2}\Big\} \geq \frac{3m-2}{m-2}.
$$
Then we can repeat the argument for the first case and choose $\de_0>0$ small enough such that
\be\label{est-Ie1-2-2}
3(\alpha-1)\Big(\frac{3(\g+\Theta)}{2(\g+\Theta)-3 }+\delta_0\Big)^{-1}-\alpha =:h_{3}(\de_0)>0.
\ee
We finally have
\be\label{est-Ie1-f}
|I_{1,\e}| \leq C\, \e^{\de_1} \big(\|\nabla\vcg{\varphi} \|_{L^{\frac{3(\g+\Theta)}{2(\g+\Theta)-3}+\delta_0}(\Om)} + \|\vcg{\varphi}\|_{L^{r_1}(\Om)}\big),
\ee
where
\be\label{def-de1}
 \de_1:=\min\left\{\frac{3(\alpha-1)}{r_1},h_{3}(\de_0)\right\}>0
\ee
 with $\de_0>0$ is chosen such that \eqref{est-Ie1-2-2} is satisfied and $1<r_1<\infty$ is determined by \eqref{def-de-r1}.

\medskip

The estimates for other $I_{j,\e}$ are similar and the results are the same (or even better). Thus we may write
\ba\label{est-Ie2}
|I_{2,\e}|+|I_{3,\e}| + |I_{4,\e}|  \leq C\, \e^{\de_1} \Big(\|\nabla\vcg{\varphi} \|_{L^{\frac{3(\g+\Theta)}{2(\g+\Theta)-3}+\delta_0}(\Om)} + \|\vcg{\varphi}\|_{L^{r_1}(\Om)}\Big).
\ea

\medskip

Summing up the estimates in \eqref{est-Ie1-f} and \eqref{est-Ie2} implies \eqref{est-re}. This completes the proof of Lemma \ref{lem:momentum}.

\end{proof}

\medskip

By \eqref{conv-1} and Lemma \ref{lem:momentum}, passing $\e\to 0$ in \eqref{eq-monmentum} gives
\be\label{eq-monmentum-limit}
\dive(\vr \vu \otimes \vu)+\nabla \overline{p(\vr, \vt)} - \dive\SSS(\vt,\nabla \vu) =  \vr \vf,\quad \mbox{in}\ \mathcal{D}'(\Om).
 \ee

To complete the proof of Theorem \ref{t_main}, we see from \eqref{w_ener_2} and \eqref{eq-monmentum-limit} that it suffices to show
\be\label{con-left}
\overline{\vr e (\vr,\vt)} = \vr e (\vr,\vt), \quad \overline{p(\vr, \vt)} = p(\vr, \vt).
\ee
Recall the formula of $e$ and $p$ in \eqref{int_e} and \eqref{press}, and recall the strong convergence of $\vte$ in \eqref{conv-1}. Thus, to show \eqref{con-left}, it is sufficient to prove the strong convergence of $\vre$. This is the main purpose of the next subsection.

\subsection{Strong convergence of the density}
\label{sub_scd}

In the theory of weak solutions of compressible Navier--Stokes equations, the strong convergence of the density is the main issue: the density has no uniform derivative estimates. While, this is nowadays well understood and the starting key point is the compactness of the so called effective viscous flux (see \cite{LI4, F-book, N-book}):

\begin{Lemma}\label{lem:flux}
Under the assumptions in Theorem \ref{t_main}, up to a subsequence, there holds for any $\psi\in C_c^\infty(\Omega)$,
\be\label{flux}
\lim_{\e\to 0}\intO{\psi \left(p( \vr_\e, \vte)-\Big(\frac{4\mu(\vte)}{3}+\nu(\vte)\Big)\dive  \vu_\e\right) \vr_\e}=\intO{\psi\left(\overline{p(\vr,\vt)}-\Big(\frac{4\mu(\vt)}{3}+\nu(\vt)\Big)\dive \vu\right)\vr}.
\ee
\end{Lemma}
\begin{proof}
The proof of Lemma \ref{lem:flux} is quite tedious but nowadays well understood. The main idea is to employ the following test functions:
$$
\psi \nabla \Delta^{-1}(1_{\Om} \vr_\e), \quad  \psi \nabla \Delta^{-1}(1_{\Om}\vr),
$$
where $\psi\in C_c^\infty(\Om)$ and $\Delta^{-1}$ is the Fourier multiplier on $ \R^3$ with symbol $-{|\xi|^{-2}}$.  We refer to Section 1.3.7.2 in \cite{N-book} or Section 10.16 in \cite{F-N-book} for more on Fourier multipliers and Riesz operators used here. We observe that
$$
(\nabla \otimes \nabla) \Delta^{-1}=\left(\mathcal{R}_{i,j}\right)_{1\leq i,j\leq 3}
$$
are the classical Riesz operators (sometimes also called double Riesz operator). Then for any $f\in L^r( \R^3),~1<r<\infty$:
\be\label{est-test-flux1}
\|(\nabla\otimes \nabla) \Delta^{-1}(f)\|_{L^r( \R^3)}\leq C(r)\, \|f\|_{L^r( \R^3)}.\nn
\ee
By the embedding theorem in homogeneous Sobolev spaces (see Theorem 1.55 and Theorem 1.57 in \cite{N-book} or Theorem 10.25 and Theorem 10.26 in \cite{F-N-book}), we have for any $f\in L^r( \R^3), \ {\supp}\, f\subset \Om$:
\ba\label{est-test-flux2}
&\|\nabla \Delta^{-1}(f)\|_{L^{r^*}( \R^3)}\leq C\, \| f \|_{L^r( \R^3)}\quad \frac{1}{r^*}=\frac{1}{r}-\frac{1}{3}, \ \mbox{if $1<r<3$},\\
&\|\nabla \Delta^{-1}(f)\|_{L^{r^*}( \R^3)}\leq C\, \| f \|_{L^r( \R^3)}\quad \mbox{for any} \ r^*<\infty,  \mbox{if $r\geq3$}.\nn
\ea
Recall the fact that
\be\label{g+t>3-1}
\gamma + \Theta  \geq \gamma + (2\g - 3) = 3\g -3 >3,
\ee
where we used the definition of $\Theta$ in \eqref{Theta}. Then by the uniform estimate for $ \vr_\e$ and $\vr$ in \eqref{b7} and \eqref{conv-1},  we have for any $1\leq r<\infty$:
\ba\label{est-test-flux3}
&\|\psi \nabla \Delta^{-1}(1_{\Om} \vr_\e)\|_{L^r(\Om)} + \|\psi \nabla \Delta^{-1}(1_{\Om}\vr)\|_{L^r(\Om)}\leq C,\\
&\|\nabla \left(\psi \nabla \Delta^{-1}(1_{\Om} \vr_\e)\right)\|_{L^{\g + \Theta}(\Om)} + \|\nabla \left(\psi \nabla \Delta^{-1}(1_{\Om}\vr)\right)\|_{L^{\g+ \Theta}(\Om)}\leq C.
\ea

Again by \eqref{g+t>3-1}, we have
$$
\frac{3(\g+\Theta)}{2(\g+\Theta)-3} < \frac{3(\g+\Theta)}{2(\g+\Theta)-(\g + \Theta)} = 3 < \g + \Theta.
$$
Thus, we can choose $\de_0>0$ in Lemma \ref{lem:momentum} small such that
$$
\g + \Theta \geq \frac{3(\g+\Theta)}{2(\g+\Theta)-3} + \de_0.
$$
Hence, by \eqref{est-re} and \eqref{est-test-flux3}, we have
\ba\label{est-test-flux4}
&|\langle {\bf r}_\e,\psi \nabla \Delta^{-1}(1_{\Om} \vr_\e)\rangle_{\mathcal{D}'(\Om),\mathcal{D}(\Om)}|\\
&\leq C\, \e^{\de_1} \left(\|\nabla \left(\psi \nabla \Delta^{-1}(1_{\Om} \vr_\e)\right)\|_{L^{3\g-3}(\Om_\e)} + \|\psi \nabla \Delta^{-1}(1_{\Om} \vr_\e)\|_{L^{r_1}(\Om)}\right)\\
&\leq C\,\e^{\de_1},\nn
\ea
which goes to zero as $\e\to 0$ due to $\de_1>0$.

\smallskip

Now we choose $\psi \nabla \Delta^{-1}(1_{\Omega} \vr_\e)$ as a test functions in the weak formulation of equation \eqref{eq-monmentum} and pass $\e\to 0$. Then we choose $\psi \nabla \Delta^{-1}(1_{\Omega}\vr)$ as a test functions in the weak formulation of \eqref{eq-monmentum-limit}. By comparing the results of theses two operations, by using the convergence results in \eqref{conv-1}, through long but straightforward calculations, we obtain that
\ba\label{flux-l1}
I:&=\lim_{\e\to 0}\intO{\psi \left(p( \vr_\e,\vte)-\Big(\frac{4\mu(\vte)}{3}+\nu(\vte)\Big)\Div  \vu_\e\right) \vr_\e}-\intO{\psi\left(\overline{p(\vr,\vt)}-\Big(\frac{4\mu(\vt)}{3}+\nu(\vt)\Big)\Div \vu\right)\vr}\\
&=\lim_{\e\to 0}\intO{ \vr_\e u_\e^i u_\e^j \psi \mathcal{R}_{i,j}(1_{\Omega}  \vr_\e)}-\intO{\vr u^i u^j \psi \mathcal{R}_{i,j}(1_{\Omega} \vr)}.
\ea

On the other hand, choosing $1_{\Omega} \nabla \Delta^{-1}(\psi \vr_\e  \vu_\e)$ as a test function in the weak formulation \eqref{w_cont_ren} with $b(\vr)=\vr$ and  $1_{\Omega} \nabla \Delta^{-1}(\psi\vr \vu)$ as a test function in the weak formulation of \eqref{con-con-3} implies
\be\label{flux-l2}
\intO{1_{\Omega} \vr_\e  u_\e^i  \mathcal{R}_{i,j}(\psi \vr_\e  \vu_\e) }=0,\quad \intO{1_{\Omega}\vr u^i  \mathcal{R}_{i,j}(\psi \vr \vu)}=0.
\ee

Plugging (\ref{flux-l2}) into (\ref{flux-l1}) yields
\ba\label{flux-l3}
&I=\lim_{\e\to 0}\intO{ u_\e^i\Big( \vr_\e u_\e^j \psi \mathcal{R}_{i,j}(1_{\Omega}  \vr_\e)-1_{\Omega} \vr_\e   \mathcal{R}_{i,j}(\psi \vr_\e  \vu_\e)\Big)}\\
&\qquad\qquad-\intO{ u^i\Big(\vr u^j \psi \mathcal{R}_{i,j}(1_{\Omega} \vr)-1_{\Omega}\vr  \mathcal{R}_{i,j}(\psi \vr \vu)\Big)}.\nn
\ea

We introduce the following lemma, which is a variant of the Div-Curl lemma. We refer to \cite[Lemma 3.4]{FNP} for its proof.
\begin{Lemma}\label{lem-ct}Let $1<p,q<\infty$ satisfy $$\frac{1}{r}:=\frac{1}{p}+\frac{1}{q}<1.$$
Suppose
\[
u_\e \to u \quad\mbox{weakly in}\quad L^p( \R^3),\quad v_\e \to v \quad\mbox{weakly in}\quad L^q( \R^3),\ \mbox{as $\e \to 0$}.
\]
Then for any  $1\leq i,j\leq 3$:
\[
u_\e \mathcal{R}_{i,j}(v_\e)-v_\e \mathcal{R}_{i,j}(u_\e) \to u \mathcal{R}_{i,j}(v)-v \mathcal{R}_{i,j}(u) \quad\mbox{weakly in}\quad L^r( \R^3).
\]
\end{Lemma}
Now, by the strong convergence of the velocity in \eqref{conv-1} and Lemma \ref{lem-ct}, our desired result \eqref{flux} follows immediately.
\end{proof}

\medskip

We rewrite \eqref{flux} into the form
$$
\intO{\psi \left(\overline{\varrho^{\gamma+1}} + \overline{\varrho^2}\vartheta-\Big(\frac{4\mu(\vte)}{3}+\nu(\vte)\Big)\overline{\varrho\dive \vu}\right)}=\intO{\psi\left(\overline{\varrho^\gamma} + \varrho \vartheta-\Big(\frac{4\mu(\vt)}{3}+\nu(\vt)\Big)\dive \vu\right)\vr}.
$$
Recall that all terms are integrable in higher power than $1$, therefore the limits exist.
This implies that 
$$
\overline{\varrho^{\gamma+1}} + \overline{\varrho^2}\vartheta-\Big(\frac{4\mu(\vt)}{3}+\nu(\vt)\Big)\overline{\varrho\dive \vu} = \varrho \overline{\varrho^\gamma} + \varrho^2 \vartheta-\Big(\frac{4\mu(\vt)}{3}+\nu(\vt)\Big)\vr \dive \vu
$$
a.e. in $\Omega$ and also
\begin{equation} \label{5.33b}
\frac{\overline{\varrho^{\gamma+1}} + \overline{\varrho^2}\vartheta}{\frac{4\mu(\vt)}{3}+\nu(\vt)}-\overline{\varrho\dive \vu} = \frac{\varrho \overline{\varrho^\gamma} + \varrho^2 \vartheta}{\frac{4\mu(\vt)}{3}+\nu(\vt)}-\vr \dive \vu
\end{equation}
a.e. in $\Omega$. Note that due to our assumptions on the viscosity coefficients, all terms are integrable over $\Omega$.

\medskip

Before formulating the last lemma, we recall one standard result (for the proof see \cite[Theorem 10.19]{F-N-book})
\begin{Lemma} \label{weak_mon}
Let $(P,G) \in C(\R) \times C(\R)$ be a
couple of non-decreasing functions. Assume that $\vr_n\in
L^1(\Omega)$ is a sequence such that
$$
\left.\begin{array}{c}
P(\vr_n) \rightharpoonup \overline{P(\vr)}, \\
G(\vr_n) \rightharpoonup \overline{G(\vr)}, \\
P(\vr_n)G(\vr_n) \rightharpoonup \overline{P(\vr)G(\vr)}
\end{array} \right\} \mbox{ in } L^1(\Omega).
$$
\begin{description}
\item[i)]
Then
$$
\overline{P(\vr)}\, \, \overline{G(\vr)} \leq
\overline{P(\vr)G(\vr)}
$$
a.e. in $\Omega$.
\item[ii)] If, in addition,
$$
G(z)=z,\quad P \in C(\R), P \mbox{ non-decreasing }
$$
and
$$
\overline{P(\vr)}\, \, \vr = \overline{P(\vr)\vr}
$$
(where we have denoted by $\vr=\overline{G(\vr)}$), then
$$
\overline{P(\vr)} = P(\vr).
$$
\end{description}
\end{Lemma}

\medskip

We now have 
\begin{Lemma}\label{lem:flux2} 
It holds $\overline {\vr^{\g+1}}=\overline {\vr^\g} \vr$ a.e. in $\Omega$. Whence $\varrho_\ep \to \varrho$ strongly in $L^1(\Omega)$ and thus also in $L^r(\Omega)$, $1\leq r<\gamma + \Theta$. 
\end{Lemma}

\begin{proof}
We follow the approach from \cite{NoPo_JDE}, the second last limit passage $\ep\to 0$ from Section 4. First, using Remark \ref{rem:ren}, we apply the renormalized continuity equation for the limit continuity equation with the function $b(\varrho) = \varrho \log \varrho$ and the test function identically equal one in $\Omega$. This leads to
$$
\int_\Omega \varrho \dive \vu \dx = 0.
$$
Similarly, using the same for the problem for $\ep >0$ and then passing $\ep \to 0$ gives
$$
\int_\Omega \overline{\varrho \dive \vu} \dx = 0.
$$
Therefore we may integrate \eqref{5.33b} over $\Omega$ to get
$$
\int_\Omega \frac{\overline{\varrho^{\gamma+1}} + \overline{\varrho^2}\vartheta}{\frac{4\mu(\vt)}{3}+\nu(\vt)} \dx  = \int_\Omega \frac{\varrho \overline{\varrho^\gamma} + \varrho^2 \vartheta}{\frac{4\mu(\vt)}{3}+\nu(\vt)} \dx.
$$
We now apply Lemma \ref{weak_mon} and see that
$$
\varrho^2 \leq \overline{\varrho^2} \qquad \text{and} \qquad \varrho \overline{\varrho^\gamma}  \leq \overline{\varrho^{\gamma+1}}
$$
a.e. in $\Omega$. Since $\vartheta >0$ a.e. in $\Omega$ (see Lemma \ref{temp_pos}), we conclude that $\overline{\varrho^{\gamma+1}} = \varrho \overline{\varrho^\gamma}$ a.e. in $\Omega$ which implies that
$$
\overline{\varrho^{\gamma}} =  \varrho^\gamma  \quad \text{a.e. in } \Omega,
$$
again by Lemma \ref{weak_mon}. Therefore, up to the choice of a subsequence, $\varrho_\ep \to \varrho$ in $L^\gamma(\Omega)$, thus also a.e. in $\Omega$ and in $L^r(\Omega)$, $1\leq r < \gamma +\Theta$. This finishes the proof of Theorem \ref{t_main}.
\end{proof}


\end{document}